\documentclass[preprint,authoryear,10pt]{elsarticle}

\usepackage{amsfonts}
\usepackage{amsmath}
\usepackage{amssymb}
\usepackage[authoryear]{natbib}

\newtheorem{theorem}{Theorem}[section]
\newtheorem{lemma}[theorem]{Lemma}

\newtheorem{corollary}[theorem]{Corollary}

\newtheorem{definition}[theorem]{Definition}

\newtheorem{remark}[theorem]{Remark}

\def \outlineby #1#2#3{\vbox{\hrule\hbox{\vrule\kern #1%
\vbox{\kern #2 #3\kern #2}\kern #1\vrule}\hrule}}%
\def \endbox {\outlineby{4pt}{4pt}{}}%
\newenvironment{proof}
{\vspace{5pt}\noindent{\bf Proof\ }}{{\hfill \endbox
}\par\vskip2\parsep}
\newenvironment{eg}
{\vspace{5pt}\noindent{\bf Example\ }}{{\hfill \endbox
}\par\vskip2\parsep}

\newcommand{\var}{{\rm{Var}}}
\newcommand{\cov}{{\rm{Cov\;}}}

\newcommand{\ep}{{\mathbb {E}}}
\newcommand{\pr}{{\mathbb {P}}}
\newcommand{\re}{{\mathbb {R}}}

\newcommand{\Z}{{\mathbb {Z}}}

\newcommand{\bdiy}{\begin{displaystyle}}
\newcommand{\ediy}{\end{displaystyle}}
\newcommand{\Po}{{\rm{Po}}}

\newcommand{\blah}[1]{}

\newcommand{\lf}{\lfloor}
\newcommand{\rf}{\rfloor}
\newcommand{\ulc}[1]{{\mbox{ULC}}(#1)}
\newcommand{\ulci}{\ulc{\infty}}

\journal{Statistics and Probability Letters}

\begin{document}
\begin{frontmatter}

\title{Bounds on the Poincar\'{e} constant under negative dependence}

\author[himr]{Fraser Daly\fnref{fn1}}
\ead{fraser.daly@bristol.ac.uk}
\author[uob]{Oliver Johnson}
\ead{o.johnson@bristol.ac.uk}

\fntext[fn1]{Corresponding author.  Tel: +44 (0)117 954 5667, Fax: +44 (0)117 928 7999.}

\address[himr]{Heilbronn Institute for Mathematical Research,
 Department of Mathematics, University of Bristol, University Walk, Bristol BS8 1TW, UK.} 
\address[uob]{Department of Mathematics, University of Bristol, University Walk, Bristol BS8 1TW, UK.}

\begin{abstract}
We give bounds on  the Poincar\'e (inverse spectral gap) constant of
a non--negative, integer--valued random variable $W$,
 under negative dependence assumptions  such as ultra log--concavity and total negative dependence.
We show that the bounds obtained compare well to others in the literature.
Examples treated include some occupancy and urn models, a random graph model and small spacings
on the circumference of a circle.   
 Applications to Poisson convergence theorems are considered. 
\end{abstract}


\end{frontmatter}

\section{Introduction}

Throughout this note we let $W$ be a random variable supported on (a subset of) $\Z ^+=\{0,1,\ldots\}$ and let $\Delta$ be the forward difference operator, so that
for any function $g:\Z ^+\mapsto\re$, $\Delta g(k)=g(k+1)-g(k)$.  The main object we wish to consider here is the (discrete) Poincar\'e constant,
given by
$$
R_W = \sup_{g\in\mathcal{G}(W)} \left\{ \frac{\ep[g(W)^2]}{\ep[\Delta g(W)^2]}  \right\}   \,,
$$
where the supremum is taken over the set
$$
\mathcal{G}(W) = \{g:\Z ^+\mapsto\re\mbox{ with } \ep[g(W)^2]<\infty \mbox{ and }\ep[g(W)]=0 \}\,.
$$

In this note we give an explicit upper bound on $R_W$ when $W$ satisfies a negative dependence assumption.
Such a bound can be used, for example, to establish Poisson convergence results, as we shall see below.
In the examples we consider in Section \ref{sec:egs} we shall see that our upper bound is easily
calculated, and often of the same order as the trivial lower bound we state in Lemma \ref{lem:lower}. 

Our work can be understood in the context of size--biasing \citep[see, for example,][]{bhj} and stochastic
ordering \citep[see, for example,][]{shaked}.
For any non--negative, integer--valued random variable $W$ with mean
$\ep W = \lambda >0$, we let $W^\star$ have the $W$--size--biased distribution, given by
\begin{equation} \label{eq:sizebias}
P(W^\star=j) = \frac{jP(W=j)}{\lambda}\,,  \mbox{\;\;\; for $j=1,2,\ldots$. }
\end{equation}
Equivalently, $W^*$ can be defined by requiring that
\begin{equation} \label{eq:altsizebias}
\ep [ W g(W)] = \lambda \ep g(W^*),
\end{equation}
for all functions $g$ for which the expectation of the left hand side exists.

Further we let $\leq_{st}$ denote the usual stochastic ordering, so that 
$X\leq_{st}Y$ if $\ep f(X)\leq\ep f(Y)$ for all increasing functions $f$. In this paper we shall consider
random variables $W$ under the assumption that $W^* \leq_{st} W+Z$ for some $Z$, the sharpness
of which can be judged from the fact that $W \leq_{st} W^*$ always.
For example, we obtain the following bound on the Poincar\'e constant:
\begin{theorem} 
\label{thm:tnd}
Let $W$ be a non--negative, integer--valued random variable with mean $\lambda$.  Suppose that 
$W^\star\leq_{st}W+1$.  Then 
\begin{equation} R_W\leq \lambda. \label{eq:rmean} \end{equation}
\end{theorem}
This theorem is implied by a stronger result, Theorem \ref{thm:ord}, which we prove in Section \ref{sec:proofs}.
In Section \ref{sec:negdep} we state several negative dependence concepts
 under which we have a finite Poincar\'e constant, and state the bounds on $R_W$ we
  obtain under these assumptions. In particular, in Corollary \ref{cor:tnd}
we show that (\ref{eq:rmean}) holds for $W$ the sum of totally negatively dependent (TND) Bernoulli
random variables, and in Corollary \ref{cor:ulc} we show (\ref{eq:rmean}) holds for $W$ ultra log-concave
of degree $\infty$.

The idea of proving such discrete Poincar\'{e} inequalities is not a new one,
with previous authors to consider the problem including
\citet{bobkov2}, \citet{gao}, \citet{klaasen}, \citet{miclo} and \citet{prakasa}.
Some comparison with the results of other authors is given in Section \ref{sec:egs}.
In this section, we also treat some examples for which bounds on the Poincar\'e constant based on 
other authors' work are not straightforward to calculate. 

The main aim of this work is to prove the new bound (Theorem \ref{thm:tnd}) on the Poincar\'{e}
constant for discrete random variables. Our approach enables us to make
new connections between the Poincar\'{e} constant and topics such as
stochastic ordering and various forms of negative dependence of random
variables. In addition, we gain a new perspective on Poisson approximation by
using our approach to give new bounds which are close
to the optimal ones.

Section \ref{sec:pois} shows how the bounds we derive may be used in proving Poisson 
convergence results.  We will assess closeness of 
 non--negative, integer--valued random variables $X$ and $Y$  using
the total variation distance 
$$
d_{TV}(\mathcal{L}(X),\mathcal{L}(Y)) = \sup_{A\subseteq\Z ^+} 
|P(X\in A) - P(Y\in A)|\,.
$$
We prove the following result in Section \ref{sec:pois}:
\begin{theorem}
\label{thm:bdtv} For any non-negative, integer--valued random variable $W$ with mean $\ep W = \lambda$ and
$\var(W)<\infty$, the total
variation distance between $W$ and a Poisson random variable with the same mean is
bounded by
\begin{equation} \label{eq:rbound}
d_{TV}(\mathcal{L}(W),\Po(\lambda)) \leq 
\frac{1 - e^{-\lambda}}{\lambda}\left\{ |\lambda -R_W| + \sqrt{R_W}\sqrt{R_W-\var(W)} \right\}\,.
\end{equation}
\end{theorem}
(Note that Lemma \ref{lem:lower} below states that $R_W \geq \var W$, so the resulting bound gives a real number).

Our results may also be used with several other probability metrics. For example, using Theorem 1.1 of 
\citet{barbourxia} we can bound the Wasserstein distance (replacing the Stein factor $\frac{1 - e^{-\lambda}}{\lambda}$
of Theorem \ref{thm:bdtv} by $1.1437/\sqrt{\lambda}$).

To illustrate Theorem \ref{thm:bdtv}, we combine it with Theorem \ref{thm:tnd} to give
\begin{equation}\label{eq:tv1}
d_{TV}(\mathcal{L}(W),\Po(\lambda)) \leq (1 - e^{-\lambda})\sqrt{1-\frac{\var(W)}{\lambda}}\,,
\end{equation}
for $W$ such that $W^\star\leq_{st}W+1$.  Note that such $W$ have $\var(W)<\lambda$.  The best known general bound under this assumption is
\begin{equation}\label{eq:tv2}
d_{TV}(\mathcal{L}(W),\Po(\lambda)) \leq (1 - e^{-\lambda})\left(1-\frac{\var(W)}{\lambda}\right)\,.
\end{equation}
See \citet{dlu}. So, for example, if $W\sim\mbox{Bin}(n,p)$ has a binomial distribution (which does indeed satisfy the assumptions of Theorem \ref{thm:tnd}, 
as we shall see later) then the upper bound in (\ref{eq:tv1}) is $\left(1-e^{-np}\right)\sqrt{p}$, while (\ref{eq:tv2}) gives the upper bound 
$\left(1-e^{-np}\right)p$.

We note also that the bound (\ref{eq:rbound}) is increasing as a function of $R_W$ for $R_W$ lying in $(\var(W),\lambda)$.  Hence, it will typically be sharpest
when we have an upper bound for $R_W$ that is close to $\var(W)$, and in the limit we recover the bound (\ref{eq:tv2}) (under the assumptions of 
Theorem \ref{thm:tnd}).

\section{Negative dependence and the Poincar\'e constant} \label{sec:negdep}

In this section we show that our methods can bound the Poincar\'{e} constant of random 
variables under several well-known
definitions of negative dependence.

\subsection{Total negative dependence}

\begin{definition} \label{def:tnd}
Consider $X_1,\ldots,X_n$ to be Bernoulli random variables and write  $W=X_1+\cdots+X_n$.
If, for each $i=1,\ldots,n$ and all increasing functions $f,g;\mathbb{Z}^+\mapsto\re$, 
$$
\cov(f(X_i),g(W-X_i))\leq0\,,
$$
then 
$X_1,\ldots,X_n$ are said to be totally negatively dependent (TND). 
\end{definition}
 \citet[Lemma 3.1]{pap} show that if 
$W = X_1 + \ldots + X_n$, where $X_1, \ldots, X_n$ are TND then
$W^\star\leq_{st}W+1$.  Using Theorem \ref{thm:tnd} we can deduce:
\begin{corollary} \label{cor:tnd}
If $W$ has mean $\lambda$ and may be written as a sum of TND Bernoulli random variables then $R_W \leq \lambda$.
\end{corollary}

\subsection{Ultra log--concavity}

\begin{definition}
A random variable $W$ with mass function $P_W$ is ultra log-concave of degree $\infty$, or
$\ulci$, if the function $P_W(x) x!$ is log-concave, or equivalently
$$
\rho_W(x) = \frac{ (x+1) P_W(x+1)}{P_W(x)}\mbox{ is non-increasing in $x$}\,.
$$
\end{definition}
This class was introduced by \citet{pemantle} and
\citet{liggett} in order to capture properties of negative dependence.
It is well-known that Poisson random variables and 
the sum of  independent Bernoulli random variables both have the $\ulci$ property.

In the language of stochastic ordering, the $\ulci$ property may be written
$W^\star\leq_{lr}W+1$, where $\leq_{lr}$ represents the likelihood ratio ordering.
Since the likelihood ratio ordering is stronger than the usual stochastic ordering 
\citep[Theorem 1.C.1]{shaked},
it follows that the assumption that $W$ is $\ulci$ is stronger than the assumption
made by Theorem \ref{thm:tnd}.  Hence we deduce the following.
\begin{corollary} \label{cor:ulc}
If $W$ is $\ulci$ with mean $\lambda$, then $R_W\leq\lambda$. \end{corollary}
Note that the ULC condition can be viewed as in the context of the
Bakry-\'{E}mery condition \citep{bakry} for continuous measures, under
which logarithmic Sobolev
inequalities are known to hold, and hence the continuous version of the  
Poincar\'{e} constant is finite (see \citet{ane} for more details).
The Bakry-\'{E}mery condition is known to hold if the relative
density $f/\phi_{1/c}$ is log-concave (in the continuous sense) and ULC requires
that $P_W/\Po(\lambda)$ is log-concave (in the discrete sense).
(Here $\phi_t$ is a normal density with mean $0$ and variance $t$ and
$\Po(\lambda)$ is a Poisson mass function with parameter $\lambda$).

\subsection{Other negative dependence assumptions}

As well as total negative dependence and ultra log--concavity, there are other well-known negative dependence assumptions which
fit into the framework of Theorem \ref{thm:ord}.  For example, we recall the definition of negative association from  
\cite{proschan}.  
\begin{definition} \label{def:na}
Random variables $X_1,\ldots,X_n$ are negatively associated if
$$
\cov \left( f(X_i,i\in\Gamma_1), g(X_i,i\in\Gamma_2) \right)  \leq 0 \,,
$$
for all increasing functions $f,g$, and all $\Gamma_1,\Gamma_2\subseteq\{1,\ldots,n\}$ with
$\Gamma_1\cap\Gamma_2=\emptyset$. 
\end{definition}
Comparison with Definition \ref{def:tnd} shows that negative association implies total negative dependence, meaning
that the TND class is a large and natural class within which our results hold. Since many authors work in terms of
negative association, we remark that our techniques can be applied directly in the NA class. That is,
if $W=X_1+\cdots+X_n$, where $X_1,\ldots,X_n$ are negatively associated, non--negative, integer--valued random variables 
then Lemma 3 of \citet{daly} shows how to construct a random variable
$Z$ such that the assumptions of Theorem \ref{thm:ord} are satisfied.  We refer the reader to that work for further details.

Another form of negative dependence is the following:
\begin{definition} \label{def:nr}
Indicators $X_1, \ldots ,X _n$ are negatively related  (NR) iif
$$ \ep [g(X_1,...X_{i-1},X_{i+1},...X_n)|X_i=1] \leq
\ep [g(X_1,...X_{i-1},X_{i+1},...X_n)],$$  
for all $i=1,\ldots,n$ and all increasing
functions $g$.
\end{definition}

This definition arises naturally in the context of certain urn and graph models -- see for example the P\'{o}lya
sampling example and random graph example \citep[Example 1]{arratia} treated in Section \ref{sec:further}. 
NR is more restrictive than NA in the sense that in the definition of negative association,
the $X_i$ don't need to be indicators. However the following results are standard:
(a) Negatively associated indicators are negatively related \citep[see P30 of][]{bhj}, (b)
Negatively related implies total negative dependence \citep[see Theorem 3.1 of][]{pap}.

\section{Examples and comparison with other results} \label{sec:egs}

We now discuss several straightforward examples, and consider how our results compare with known bounds. 
First we mention a trivial
lower bound, the direct equivalent of Theorem 2(vi) of \citet{borovkov2}:

\begin{lemma} \label{lem:lower}
For any non--negative, integer--valued random variable $W$ with finite variance:
$ R_W \geq \var(W).$
\end{lemma}
\begin{proof}
Writing $\lambda = \ep W$,
consider the function $g(x) = x - \lambda$, so that $\Delta g(x) = 1$.
We know that
$\ep[g(W)^2] = \var(W)$ and $\ep[\Delta g(W)^2] = 1$. 
It follows that 
$R_W = \sup_g (\ep g(W)^2)/(\ep (\Delta g)(W)^2) \geq \var(W)$.
\end{proof}

\begin{remark} \label{rem:bdscompatible} Note that Theorem
\ref{thm:tnd} and Lemma \ref{lem:lower} together imply that for 
$W$ such that $W^\star\leq_{st}W+1$ we have $\var(W) \leq R_W \leq \lambda$.
These bounds are compatible, in the sense that for $W$ satisfying this
condition, taking $g(x) = x-1$ in Equation (\ref{eq:altsizebias}), the fact that  $\ep W(W-1) = \ep [W g(W)]
= \lambda \ep g(W^\star)
\leq \lambda \ep g(W+1) = \lambda^2$  . (The inequality follows by the definition of $\leq_{st}$).
Equivalently we write that $\var(W) \leq \lambda$.
\end{remark}

\subsection{Comparison with other results}

Note that our bounds can be contrasted with results such as Propositions 1 and 2 of \citet{miclo},
which give upper and lower bounds on the Poincar\'e constant that differ by a constant multiplicative factor,
as opposed to the additive gap found here. 

In comparing our results with other authors to have considered the discrete
Poincar\'{e} constant, we shall see that other methods
can significantly
overestimate the exact value of the constant.  Using a generalization of Cheeger's inequality, we have the following bound. 

\begin{theorem} \label{thm:cheeger}
Let $W$ be a non--negative, integer--valued random variable with log-concave mass function $P_W$, then
$$ R_W \leq 4\left(\frac{1 - P_W(0)}{P_W(0)}\right)^2\,.$$
\end{theorem}
\begin{proof} 
Theorem 2.1 of \citet{lawler} (a generalization of Cheeger's inequality) states that 
if, there exists $c$ such that the ratio
$$ r(u) := \frac{\sum_{y > u} P_W(y)}{ P_W(u)} \leq c \mbox{\, for all $u \geq 0$},$$
then $R_W \leq 4c^2$. In particular, if $P_W$ is log-concave
(a weaker restriction than ultra log-concave) then rearranging shows that 
$r(u)$ is decreasing in $u$, so 
taking $c = r(0) =  (1-P_W(0))/P_W(0)$, the result holds.
\end{proof}
Note also that other authors, such as \citet{bobkov2}, give necessary and sufficient conditions 
for the discrete Poincar\'{e} constant to be finite
without giving explicit upper bounds on the implied Poincar\'{e} constant.

We first consider the examples of Poisson random variables and sums of independent Bernoulli random variables.  In these
examples it is straightforward to compare the bounds given by our results and the work of other authors.  In subsequent work
we consider examples for which the bound given by Theorem \ref{thm:cheeger} is less straightforward
to evaluate.

\begin{eg}
Let $W\sim\Po(\lambda)$ have a Poisson distribution.  It is straightforward to see that $W^\star$ and $W+1$
are equal in distribution, and so combining Theorem
\ref{thm:tnd} and Lemma \ref{lem:lower} we immediately return the well--known result \citep[see][]{klaasen} that $R_W=\lambda$. 
(Alternatively, we can deduce that $R_W \leq \lambda$ using  Corollary \ref{cor:ulc}
since $W$ is $\ulci$). 

Cheeger's inequality, Theorem \ref{thm:cheeger}, implies that $R_W \leq 4(e^\lambda-1)^2$,
which is clearly far from optimal in the Poisson case.
\end{eg}

\begin{eg}
Let $W=X_1+\cdots+X_n$, where $X_1,\ldots,X_n$ are independent Bernoulli random variables with
$P(X_j=1)=p_j$.  Clearly $X_1,\ldots,X_n$ are 
 TND, and so combining Theorem \ref{thm:tnd} and Lemma \ref{lem:lower}
we have that
$$
\sum_{j=1}^n p_j(1-p_j) \leq R_W \leq \sum_{j=1}^n p_j\,.
$$
 This may be
compared to the upper bound of Theorem \ref{thm:cheeger}. In the binomial case where $p_i \equiv p$,  
this gives $R_W\leq 4((1-p)^n-1)^2$.

We refer the reader also to Section 1.4 of \citet{ane}, where another upper bound is derived, but using
a different definition of the Poincar\'e constant.  In that work, the endpoints of the support of $W$ are 
identified, so their differencing operator is not equal to our operator $\Delta$.  
\end{eg}

\subsection{Further examples and applications} \label{sec:further}

We turn our attention now to some further examples in which we may apply Theorem \ref{thm:tnd}.

\begin{eg}
Let $W$ have a hypergeometric distribution, so that if we distribute $m$ balls into $N$ urns
(each with capacity for up to one ball), $W$ counts the number of the first $n$ of these
urns which are occupied.  We may write $W=X_1+\cdots+X_n$, where $X_j$ is an indicator that
the $j$th urn is occupied.  The random variables $X_1,\ldots,X_n$ are TND:
see \citet[Section 6.1]{bhj}.  Hence, by Theorem \ref{thm:tnd} and Lemma \ref{lem:lower},
$$
\frac{mn(N-n)}{N(N-1)}\left(1-\frac{m}{N}\right) \leq R_W \leq \frac{mn}{N}\,.
$$
In particular, if $m=O(N)$ and $n=O(N)$ then $R_W=O(N)$.  See also \citet{gao}.
\end{eg}

\begin{eg}
Suppose we have $n$ urns into which we distribute $m_n=\lf tn^{1-1/c} \rf$ balls, for some  constants 
$c\in\{2,3,\ldots\}$ and $t>0$.  Let $\mu=t^c/c!$ and let $W$ count the
number of urns with at least $c$ balls.  \citet{pap} show that $W$ may be 
written as a sum of TND Bernoulli random
variables and, furthermore, that $\var(W)\geq\mu+O(n^{-1/c})$ and $\ep W = \mu+O(n^{-1/c})$ -- see
(4.9) and Remark 4.1(a) of \citet{pap}.  Combining these results with our Theorem \ref{thm:tnd} and Lemma \ref{lem:lower}
we have that $R_W=\mu+O(n^{-1/c})$.  
\end{eg}

\begin{eg} 
Consider P\'olya sampling.  We have an urn initially containing $N$ balls of $n$ different colours, with $m_i$ balls of colour $i$.
At each step, we draw a ball, note its colour and return it to the urn together with an additional ball of the same colour.  We repeat for a total
of $r$ draws and let $W$ count the number of colours not drawn during this process.  We write $W=X_1+\cdots+X_n$, where $X_j$ is an indicator that no ball
of colour $j$ is seen during the $r$ draws.  Since $X_1,\ldots,X_n$ are negatively related (see Definition \ref{def:nr}), they
 are also TND: see Section 6.3 of \citet{bhj}.  It is 
straightforward to see that for $j\not=k$
$$
p_j = \ep X_j = \frac{\binom{N-m_j+r-1}{r}}{\binom{N+r-1}{r}}\\, \hspace{20pt}\mbox{and}\hspace{20pt} p_{jk} = \ep[X_jX_k] = \frac{\binom{N-m_j-m_k+r-1}{r}}{\binom{N+r-1}{r}}\\.
$$ 
Thus, using Theorem \ref{thm:tnd} and Lemma \ref{lem:lower} we have that
$$
\sum_{j=1}^np_j(1-p_j) + \sum_{j=1}^n\sum_{k\not=j}(p_{jk}-p_jp_k) \leq R_W \leq \sum_{j=1}^np_j\\.
$$
\end{eg}

\begin{eg} 
We treat Example 1 from \citet{arratia}.  Consider the $n$ dimensional cube $\{0,1\}^n$ with each of the 
$n2^{n-1}$ edges independently assigned one of two directions with equal probability.  Let $W$ count the number
of vertices at which all $n$ incident edges are directed inward.  Then $W=X_1+\cdots+X_{2^n}$, where $X_j$ is an indicator that
vertex $j$ has all its incident edges directed inward.  The random variables $X_1,\ldots,X_{2^n}$ are negatively related
(see Definition \ref{def:nr}), and thus
also TND.  Clearly $\ep X_j=2^{-n}$ and  
\begin{equation*}
 \ep[X_jX_k] = \left\{
\begin{array}{ll}
0 & \mbox{ if vertices $j$ and $k$ share a common edge,} \\
\ep X_j\ep X_k & \mbox{ otherwise,} \\
\end{array} \right.  \end{equation*}
for $j\not=k$.  Hence, $\ep W=1$, $\mbox{Var}(W)=1-(2n+1)2^{-n}$ and, by Theorem \ref{thm:tnd} and Lemma \ref{lem:lower},
$1-(2n+1)2^{-n}\leq R_W\leq1$.  
\end{eg}

\begin{eg}
Suppose we distribute $n$ points uniformly on the circumference of a circle of radius
$(2\pi)^{-1}$.  Let $S_1,\ldots,S_n$ be the arc--length distances between successive points and define
$X_j=I(S_j<a)$ for some $a>0$, the indicator that $S_j$ falls below the threshold $a$.  
Then $X_1,\ldots,X_n$ are negatively related and thus TND \citep[see Section 7.1 of][]{bhj}.  
From calculations by \citet[Section 7.2]{bhj} we have that
$\ep W=n(1-(1-a)^{n-1})$ and $\mbox{Var}(W)\geq(1-2na)\ep W$.  Hence, by
Theorem \ref{thm:tnd} and Lemma \ref{lem:lower},
$$
(1-2na)\ep W \leq R_W \leq \ep W\,.
$$
Corollary 7.B.1(a) of \citet{bhj} shows that if $\lim\inf\ep W>0$, the distribution of $W$ converges to
that of a Poisson random variable if and only if $na\rightarrow0$.  This is closely related to the bounds
we have obtained above on $R_W$.  This connection will be further explored in Section \ref{sec:pois}.
\end{eg}

\section{Poisson convergence and the Poincar\'e constant} \label{sec:pois}

\citet{prakasa} show that the property $R_W = \var(W)$ characterizes the
Poisson distribution (up to integer shifts). Indeed, if $\pr(W = 0) > 0$,
then Theorem 1 of \citet{prakasa} shows that
$R_W = \var(W)$ implies that $W$ is Poisson distributed,
and hence we may write $R_W = \var(W) = \ep W$.  We will show here how
closeness of $R_W$ to $\var(W)$ and $\ep W$
implies closeness of $W$ to a Poisson random variable, noting that we do indeed need
both of these conditions to guarantee that we are close to a Poisson distribution (rather than a shifted
Poisson distribution). Our work is motivated by that of \citet{utev} relating analogous
characterisations for the
normal and Poisson distributions to results in convergence to those
distributions.
The results of this section serve to illustrate the benefit of sharp bounds
on the Poincar\'e constant.

We begin by proving Theorem \ref{thm:bdtv}. Many of the definitions and equalities used in the proof are drawn from Stein's method for Poisson approximation.  
See \citet{bhj} for an introduction to these ideas.   

\begin{proof}{{\bf of Theorem \ref{thm:bdtv}}}
For any non--negative, integer--valued random variable $W$ and $g\in\mathcal{G}(W)$ we have, by the definition of $R_W$ and since $\ep g(W)=0$,
\begin{equation}\label{pf3a}
\mbox{Var}(g(W)) \leq R_W \ep[\Delta g(W)^2]\,,
\end{equation}
Writing $\lambda$ for $\ep W$, we apply this with the choice 
\begin{equation} \label{pf3b}
g(x) = x + yf_A(x) - \lambda - y\ep f_A(W)\,,
\end{equation}
where $y\in\re$ will be determined later, $A\subseteq\Z ^+$ and $f_A:\Z ^+\mapsto\re$ solves
the Chen-Stein equation
$$
I(x\in A) - P(\Po(\lambda)\in A) =  \lambda f_A(x+1)-xf_A(x)\,,
$$
so that 
\begin{equation}\label{pf3c}
d_{TV}(\mathcal{L}(W),\Po(\lambda))=\sup_{A\subseteq\Z ^+}| \lambda \ep f_A(W+1)-\ep[Wf_A(W)]|\,.
\end{equation}  
We will need the property that
\begin{equation}\label{pf3d}
\sup_{A\subseteq\Z ^+}\sup_{x\in\Z^+}|\Delta f_A(x)| \leq \frac{1 - e^{-\lambda}}{\lambda}\,,
\end{equation} 
see \citet[Lemma 1.1.1]{bhj} and \citet[Equation (1.4)]{pap}.

Now, applying (\ref{pf3a}) with the choice (\ref{pf3b}) gives us that $\alpha y^2 +2\beta y +\gamma\geq0$, where
\begin{eqnarray*}
\alpha &=& R_W\ep[\Delta f_A(W)^2] - \var (f_A(W))\,,\\
\beta &=& R_W\ep[\Delta f_A(W)]-\ep[Wf_A(W)] + \lambda \ep f_A(W)\,,\\
\gamma &=& R_W-\var (W)\,.
\end{eqnarray*}
Since this quadratic function can have at most one real root, $|\beta|\leq\sqrt{\alpha\gamma}$.  That is
\begin{eqnarray*}
\lefteqn{ |(\ep W)\ep f_A(W+1) - \ep[Wf_A(W)] + (R_W-\ep W)\ep\Delta f_A(W)| } \\
& \leq & \sqrt{R_W-\var (W)}\sqrt{R_W\ep[\Delta f_A(W)^2]-\var (f_A(W))}\\ 
& \leq & \sqrt{R_W-\var (W)} \sqrt{R_W} \frac{(1 - e^{-\lambda})}{\lambda} \,,
\end{eqnarray*}
where this last inequality follows from (\ref{pf3d}).  Combining this with the triangle inequality, (\ref{pf3c}) and (\ref{pf3d}) we obtain Theorem \ref{thm:bdtv}.
\end{proof}

Note that we can treat the RHS of Equation (\ref{eq:rbound}) as a function
of $R_W$. When $\var(W) \leq R_W \leq \lambda$ 
(for example when $W^\star\leq_{st}W+1$, as discussed in Remark 
\ref{rem:bdscompatible}), this function is increasing in $R_W$ over this range
(since it is concave and increasing at $R_W = \lambda$). Hence, providing
tighter bounds from above on $R_W$ would give tighter bounds on the 
rate of Poisson convergence.

In view of Corollary \ref{cor:tnd} and Corollary \ref{cor:ulc} we obtain the following:
\begin{corollary}
Let $\{W_n:n\geq1\}$ be non--negative, integer--valued random variables such that 
 $\lim_{n\rightarrow\infty}|\ep W_n - \var (W_n)|=0$.
Suppose that any of the following three conditions hold:
\begin{enumerate}
 \item \label{itm:cond1} $\lim\sup_n R_{W_n}<\infty$ and
$\lim_{n\rightarrow\infty}|R_{W_n}-\var (W_n)|=0$,
 \item \label{itm:cond2} $\lim\sup_n \ep W_n < \infty$ and  for each $n$, 
$W_n$ may be written as a sum of TND Bernoulli random variables; or
 \item \label{itm:cond3} $\lim\sup_n \ep W_n < \infty$ and  for each $n$, $W_n$ is $\ulci$.
\end{enumerate}
Then $d_{TV}(\mathcal{L}(W_n) , \Po (\ep W_n))\rightarrow0$ as $n\rightarrow\infty$.
\end{corollary}
\begin{proof} Note that since $f(t) = (1- e^{-t})/t$ is decreasing in $t$, we can bound it by
$f(t) \leq \lim_{t \rightarrow 0} f(t) = 1$. Hence in Theorem \ref{thm:bdtv}, we need only
show that the term in braces converges to zero. Under Condition \ref{itm:cond1} this
is automatic. Further, note that Conditions \ref{itm:cond2} and \ref{itm:cond3} each imply
that $\var(W_n) \leq R_{W_n} \leq \ep W_n$.  This follows from Lemma \ref{lem:lower} and from Corollary \ref{cor:tnd}  and
\ref{cor:ulc} respectively.  Thus, by a sandwich argument, we know that Conditions
\ref{itm:cond2} and \ref{itm:cond3} each imply Condition \ref{itm:cond1}.
\end{proof} 

\section{Proof of Theorem \ref{thm:tnd}} \label{sec:proofs}

In this section we prove the following result, which gives Theorem \ref{thm:tnd} as an immediate corollary.
\begin{theorem} \label{thm:ord}
Let $W$ be a non--negative, integer--valued random variable with mean $\lambda$, and let $Z\geq1$ 
be a random variable defined on the same space as $W$ 
 such that $W^\star\leq_{st}W+Z$.  Then for any $g\in\mathcal{G}(W)$, 
$$
 \var \; g(W) = \ep[g(W)^2] \leq \lambda \sum_{j=0}^\infty \Delta g(j)^2 P(j-Z<W\leq j)\,.
$$
\end{theorem}
Our proof will make use of Klaassen's kernel function, given by equation (2.17) of \citet{klaasen}:
\begin{equation} \label{eq:klaker}
\chi(i,j) =  I( \lf x_0 \rf \leq j < i) - I( i \leq j < \lf x_0 \rf) - (x_0 - \lf x_0 \rf) I( j = \lf x_0 \rf)\,,
\end{equation}
for some $x_0\in\re$.  We begin with the following lemma.

\begin{lemma} \label{lem:pf1}
Let $W$ be a non--negative, integer--valued random variable.  Then for any $g\in\mathcal{G}(W)$
and any $x_0 \in \re$,
\begin{equation}
\label{eq:doublesum}
 \ep[g(W)^2] \leq \sum_{j=0}^\infty \Delta g(j)^2 \sum_{i=0}^\infty (i-x_0)P(W=i)\chi(i,j) \,.
\end{equation}
\end{lemma}
\begin{proof}
For any given integer $i$, by considering the cases
$\{ i < x_0 \}$, $\{ i > x_0 \}$ and $\{ i = x_0 \}$ separately,
 we deduce that for any function $h$:
\begin{equation*}
 \sum_{j=0}^{\infty} \chi(i,j) h(j) = \left\{
\begin{array}{ll}
- \sum_{j=i}^{\lf x_0 \rf-1} h(j) + (\lf x_0 \rf - x_0) h( \lf x_0 \rf) & \mbox{ for $i < \lf x_0 \rf$,} \\
(\lf x_0 \rf - x_0) h( \lf x_0 \rf) & \mbox{ for $i = \lf x_0 \rf$,} \\
\sum_{\lf x_0 \rf}^{i-1} h(j) + (\lf x_0 \rf - x_0) h( \lf x_0 \rf) & \mbox{ for $i > x_0$.} \\
\end{array} \right.  \end{equation*}
Taking $h \equiv \Delta g$ we deduce that
$\sum_{j=0}^{\infty} \chi(i,j) \Delta g(j) = g(i) - g^*$, where $g^* = g( \lf x_0 \rf) +
\Delta g(\lf x_0 \rf)( x_0 - \lf x_0 \rf)$. In particular, taking
$h(j) \equiv 1$ we deduce that
$ \sum_{j=0}^{\infty} \chi(i,j) = (i-x_0)$.
Observe that by the Cauchy-Schwarz inequality this means that
\begin{eqnarray} 
\left( g(i) - g^* \right)^2 & = &
\left( \sum_{j=0}^{\infty} \chi(i,j) \Delta g(j) \right)^2
\leq \left( \sum_{j=0}^{\infty} \chi(i,j)  \right) \left( \sum_{j=0}^{\infty} \chi(i,j) \Delta g(j)^2 \right) \nonumber \\
& = & (i- x_0) \left( \sum_{j=0}^{\infty} \chi(i,j) \Delta g(j)^2 \right). \label{eq:cs} 
\end{eqnarray}
Although $\chi(i,j)$ is a signed measure on $j$, the use of the Cauchy--Schwarz inequality is justified since
$\chi(i,j)$ has constant sign for any given $i$.  If $i\geq\lf x_0\rf$ then $\chi(i,j)\geq0$ for all $j$,
and otherwise $\chi(i,j)\leq0$ for all $j$.

The lemma follows on combining (\ref{eq:cs}) with the observation that for all $g\in\mathcal{G}(W)$
$$
\ep[g(W)^2] \leq  \sum_{i=0}^{\infty} P(W=i) \left( g(i) - g^*\right)^2\,,
$$
and reversing the order of summation in the resulting expression.
\end{proof}

Theorem \ref{thm:ord} follows immediately from Lemma \ref{lem:pf1}.  To see this, choose $x_0=\ep W=\lambda$.
Then, using the definition of the size--bias transform (\ref{eq:altsizebias}), for a fixed $j \in \Z$
the inner sum in Equation (\ref{eq:doublesum})
can be expressed as
\begin{eqnarray}
\ep W \chi(W,j) - \lambda \ep \chi(W,j) & = & \lambda \left( \ep \chi( W^*, j) - \ep \chi(W,j) \right) 
\nonumber \\
& \leq & \lambda \ep \left( \chi(W+Z,j) - \chi(W,j) \right), \label{eq:manip}
\end{eqnarray}
using the stochastic ordering assumption of Theorem \ref{thm:ord}, and the fact that 
$\chi(i,j)$ is increasing in $i$ for fixed $j$. Using (\ref{eq:klaker}), and assuming that
$j \geq \lf x_0 \rf$ first, we observe that for any $w,z \in \Z$,
\begin{equation} \label{eq:kerdiff}
\chi( w+z, j) - \chi(w,j) = I(j < w+z) - I(j < w) = I( w \leq j < w+z). \end{equation}
Similarly, Equation (\ref{eq:kerdiff}) also holds in the case $j < \lf x_0 \rf$. Substituting
this in (\ref{eq:manip}) we obtain $\lambda \pr(W \leq j< W +Z) = \lambda \pr( j - Z < W \leq j)$
as required to complete the proof.

\section*{Acknowledgement}

We wish to thank anonymous referees for pointing us towards useful references, and for comments that helped improve the presentation of this work.

\end{document}